\documentclass[11pt]{article}
\usepackage[utf8]{inputenc}
\usepackage{amsmath,amsthm,amssymb,mathtools,twoopt}
\mathtoolsset{showonlyrefs}

\usepackage[bordercolor=orange,linecolor=orange,backgroundcolor=orange!30,textsize=footnotesize]{todonotes}
\usepackage{tikz}
\usetikzlibrary{calc, intersections, shapes}
\usetikzlibrary{positioning,automata,backgrounds}
\tikzstyle{vertex}=[circle, fill, draw, inner sep=1pt]
\colorlet{lightgray}{black!15}

\newtheorem{thm}{Theorem}
\newtheorem*{thm*}{Theorem}

\newtheorem{prop}[thm]{Proposition}
\newtheorem*{prop*}{Proposition}
\newtheorem{conj}[thm]{Conjecture}
\newtheorem{claim}[thm]{Claim}
\newtheorem{question}[thm]{Question}

\def\E{\mathbb{E}}

\def\E{\mathbb{E}}
\def\R{\mathbb{R}}

\def\eps{\epsilon}

\def\1{\mathbf{1}}

\def\C{\mathcal {C}}
\def\lam {\lambda}
\def\Lam{\Lambda}
\def\tce{t_c + \eps}
\def\tce2{t_c + \frac{\eps}{2}}

\theoremstyle{remark}
\newtheorem*{remark}{Remark}

\DeclareMathOperator{\supp}{supp}

\newcommand{\pfrac}[2]{\parens*{\frac{#1}{#2}}}

\newcommand{\set}[1]{\left\{ #1 \right\}}

\newcommand{\popt}{{p^*}}
\newcommand{\Lopt}{\Lambda^*}

\newcommandtwoopt{\Kdr}[2][d][r]{K_{{#2} \times {#1}}}
\newcommandtwoopt{\Hdr}[2][d][r]{H({[#2]}^{#1})}

\newcommand{\eqdef}{=\vcentcolon}

\def\iw{i_r}
\def\is{i_2}
\def\hmod{H^{\mathrm{mod}}}

\DeclarePairedDelimiter{\size}{\lvert}{\rvert}

\DeclarePairedDelimiter{\parens}{(}{)}

\newcommand{\defeq}{\overset{\rm{def}}{=}} 

\begin{document}
\title{On the number of independent sets in uniform, regular, linear hypergraphs}
\author{Emma Cohen\thanks{Center for Communications Research, Princeton.} \and Will Perkins\thanks{Department of Mathematics, Statistics, and Computer Science, University of Illinois at Chicago; supported in part by NSF grants DMS-1847451 and CCF-1934915.} \and Michail Sarantis\thanks{School of Mathematics, Georgia Institute of Technology.} \and Prasad Tetali\thanks{School of Mathematics and School of Computer Science, Georgia Institute of Technology; supported in part by NSF grants DMS-1407657 and DMS-1811935.}}

\maketitle

\begin{abstract}
  We study the problems of bounding the number weak and strong independent sets in $r$-uniform, $d$-regular, $n$-vertex  linear hypergraphs with no  cross-edges.  In the case of weak independent sets, we provide an upper bound that is tight up to the first order term for all (fixed) $r\ge 3$, with  $d$ and $n$ going to infinity. In the case of strong independent sets, for $r=3$,  we provide an upper bound that is tight up to the second order term, improving on a result of Ordentlich-Roth (2004). The tightness in the strong independent set case is established by an explicit construction of a $3$-uniform, $d$-regular, cross-edge free, linear hypergraph on $n$ vertices which could be of interest in other contexts. We leave open the general case(s) with some conjectures. Our proofs use the occupancy method introduced by Davies, Jenssen, Perkins, and Roberts (2017). 
  
\end{abstract}

\section{Introduction}

A classic result in the extremal theory of bounded-degree graphs is the result of Jeff Kahn~\cite{Kahn2001} that a disjoint union of copies of the complete $d$-regular bipartite graph ($K_{d,d}$) maximizes the number of independent sets over all $d$-regular bipartite graphs on the same number of vertices. The result was later extended to all $d$-regular graphs by Yufei Zhao~\cite{zhao2010number}.   

\begin{thm}[Kahn, Zhao]
\label{thmKahn}
Let $i(G)$ denote the total number of independent sets of a graph $G$.  For all $d$-regular graphs $G$,
\begin{align}
\frac{\log i(G)}{|V(G)|} &\le \frac{\log i(K_{d,d})}{2d}\,. 
\end{align}
\end{thm}
The logarithmic formulation of the theorem is equivalent to that in the preceding paragraph since $i(G)$ is multiplicative over unions of disjoint graphs. 

This result has led to many extensions and generalizations; Galvin and Tetali~\cite{galvin2004weighted} extended Kahn's result from counting independent sets to counting graph homomorphisms (and weighted independent sets).  For a recent survey of extremal results for regular graphs see~\cite{yufeiSurvey}.

The broad question we aim to address here is what are the possible generalizations of Theorem~\ref{thmKahn} to hypergraphs?

A hypergraph $G = (V,E)$ is a set of vertices $V$ and a collection of edges $E$ with each edge a subset of $V$.  A hypergraph is  \emph{$r$-uniform} if each edge contains exactly $r$ vertices.  The degree of a vertex $v \in V$ is $d(v) = |\{ e \in E : v \in e \}|$.   We say $u\sim v$ ($u$ is a \emph{neighbor} of $v$) if $u\neq v$ and there is some edge $e\in E(H)$ with $\{u,v\}\subseteq e$.   The \emph{neighborhood} of $v$ is $N(v) = \{u\in V(H)\mid u\sim v\}$.

A hypergraph is \emph{linear} if each pair of distinct vertices appear in at most one common edge.    A \emph{cross edge} in the neighborhood of a vertex $v$ is an edge $e$ that contains two neighbors $u_1, u_2$ of $v$ but not $v$ itself.  A hypergraph is \emph{cross-edge free} is it contains no cross-edges.  In an ordinary graph (a $2$-uniform hypergraph) being cross-edge free is being triangle free.  

A \emph{$k$-independent set} in a hypergraph is a a set $I \subseteq V(G)$ so that $|I \cap e| < k$ for all $e \in E(G)$.  Let $\mathcal I_k(H)$ be the set of all $k$-independent sets of a hypergraph $H$, and $i_k(G) = |\mathcal I_k(H)|$.  We refer to a $2$-independent set as a \emph{strong} independent set, and an $r$-independent set in an $r$-uniform hypergraph as a \emph{weak} independent set, and for simplicity we focus mainly on these two cases.

The main question we consider here is the generalization of Theorem~\ref{thmKahn} to linear hypergraphs.
\begin{question}
\label{qMainQ}
Which $d$-regular, $r$-uniform, linear hypergraph on a given number of vertices has the most $k$-independent sets?
\end{question}

\subsection{Strong independent sets}

Apart from the trivial cases $r=1$ or $d=1$, a tight answer to Question~\ref{qMainQ} is known in only two cases: Theorem~\ref{thmKahn} gives the answer  for $r=2$ (the case of ordinary graphs for which strong and weak independent sets coincide): for every $2$-uniform, $d$-regular hypergraph $G$ on $n$ vertices,
\begin{align}
\frac{\log \is(G)}{n} &\leq \frac{\log \is(K_{d,d})}{2d} = \frac{\log(2^{d+1}-1)}{2d} = \frac {1}{2}  +\frac{1}{2d} - \exp( - \Theta(d))\,.
\end{align}
Here and in what follows we write $\log x $ for $\log_2 x$ and $\ln x$ for the natural logarithm.  We also use standard asymptotic notation $O(\cdot), \Omega(\cdot), \Theta(\cdot)$, as $r, d \to \infty$.  A function $f(r,d) = O(g(r,d))$ if there exists a constant $C$ so that $f(r,d) \le C g(r,d)$ for all $r,d$.

The following result~\cite{davies2015} answers the question for $d=2$.  As originally phrased, the theorem states that a union of copies of $K_{r,r}$ maximizes the number of matchings of any $r$-regular graph on the same number of vertices.  However, for any graph $G$ we can define a $2$-regular, linear hypergraph $G^T$ with a vertex in $G^T$ for every edge of $G$ and an edge for every vertex of $G$, comprising of all of its incident edges (we choose this notation since the transformation transposes the edge-vertex incidence matrix).  Then $G\mapsto G^T$ is a bijection between $r$-regular ($2$-uniform, simple) graphs on $n$ vertices and $2$-regular, $r$-uniform, linear hypergraphs on $rn/2$ vertices, and matchings in $G$ correspond naturally to strong independent sets in $G^T$. Thus one of the results of \cite{davies2015} can be equivalently phrased as:

\begin{thm}[Davies, Jenssen, Perkins, Roberts]
	\label{thmMatch}
	For any $2$-regular, $r$-uniform hypergraph $G$,
	\begin{align}
	\frac{\log \is(G)}{|V(G)|} &\le \frac{\log \is(K_{r,r}^T)}{r^2} =  \Theta\left(\frac{\log r}{r} \right) \,. 
	\end{align}
\end{thm}
In other words, for strong independent sets in 2-regular hypergraphs the maximizing hypergraph is $K_{r,r}^T$, the $r\times r$ grid.

Prior to Kahn's work, when the case of ordinary graphs was still unsettled, Ordentlich and and Roth~\cite{Ordentlich2004} gave a general bound for the number of strong independent sets ($k=2$) in regular, uniform, linear hypergraphs.

\begin{thm}[Ordentlich and and Roth]
	\label{thmOR}
	For every $r$-uniform, $d$-regular, linear hypergraph $G$ on $n$ vertices,
	\[\frac{\log \is(G)}{n} \leq \frac{1}{r} + O\left(\frac{\log^2(rd)}{rd}\right).\]
\end{thm}

Their interest in the problem was motivated by understanding the number of independent sets in the Hamming graph $H(n,q)$ with vertex set $\{0,1,\ldots, q-1\}^n$ and edges between vectors at Hamming distance $1$. (They were particularly interested in $q>2$, since much more precise information  was  already known about the $q=2$ case~\cite{korshunov1983number,sapozhenko1989number}). As observed in \cite{Ordentlich2004}, a subset of the Hamming graph $H(n,q)$ is an independent set if and only if it is also a (strong) independent set in the $q$-uniform, $n$-regular linear hypergraph with the same vertex set as $H(n,q)$ and with hyperedges being the subsets of vertices that agree in all but one coordinate. Thus Theorem~\ref{thmOR} gives corresponding bounds on the number of independent sets in the Hamming graph for all $q\ge 2$.

We conjecture that the second-order term in Theorem~\ref{thmOR} can be improved.
\begin{conj}
	\label{conj:strong}
	Let $H$ be an $r$-uniform, $d$-regular, linear hypergraph on $n$ vertices.  Then
	\begin{align}
	\frac{\log \is(G)}{n} &\le  \frac{1}{r} + O\left(\frac{\log r}{ rd}\right).
	\end{align}
\end{conj}
Our first main result is to confirm this for the cases $r=3$ with the additional assumption of cross-edge freeness. 
\begin{thm}
	\label{thm:3unif}
	Let $G$ be a $3$-uniform, $d$-regular, linear hypergraph on $n$ vertices without cross-edges.  Then 
	\begin{align}
	\frac{\log \is(G)}{n} &\le \frac{1}{3} + O\left(\frac{1}{d}\right).
	\end{align}
\end{thm}
In Section~\ref{secTightExamples} we will show that the dependence on $d$ in the second-order term is best possible.

The proof of Theorem~\ref{thm:3unif} is in fact more general and gives an upper bound on the independence polynomial,
\begin{align}
Z_G(\lam) &= \sum_{I \in \mathcal I_2(G)} \lam^{|I|}  \, ,
\end{align}
for all values of $\lam >0$. The function $Z_G(\lam)$ is  known in statistical physics as the partition function of the hard-core model. We can recover $\is(G)$ by taking $\lam=1$.   Both Theorems~\ref{thmKahn} and~\ref{thmMatch} also hold at this level of generality; that is, the normalized log partition function is maximized by $K_{d,d}$ and $K_{r,r}^T$ respectively for all values of $\lam>0$.  We discuss the hard-core model and the method of proof in further detail in Section~\ref{secHardCore}.  While we believe the method of proof can be extended to additional small cases of $r$,  additional insight would be  required to push the technique to work for  $r>7$. For this reason, we restrict our attention to $r=3$ in the current presentation.

\subsection{Weak independent sets}
\label{subsecWeak}

Next we consider $r$-independent sets in $r$-uniform hypergraphs.  Recently Balabonov and Shabanov~\cite{balobanov2018number} used the  method of hypergraph containers (\cite{saxton2015hypergraph,balogh2015independent}) to give a general upper bound on $i_r$ for $r$-uniform, linear, regular hypergraphs.

\begin{thm}[Balabonov, Shabanov]
	\label{thm:shabanov}
	Suppose $G$ is a linear, $r$-uniform, $d$-regular hypergraph on $n$ vertices.  Then for $2 \le j \le r$, 
	\[ \frac{\log i_j(G) }{n} \le \frac{j-1}{r} +   O( (\log d)^{2 (j-1)/j}  \cdot d^{-1/j}) \, . \]
\end{thm}

The case of weak independent sets, $j=r$, was implicit in~\cite{saxton2015hypergraph}.

Our second main result is an improved upper bound on the number of weak independent sets in cross-edge free hypergraphs.
\begin{thm}
	\label{thm:weak}
	Suppose $G$ is a linear, $r$-uniform, $d$-regular hypergraph on $n$ vertices with no cross-edges.  Then 
	\begin{align}
	\frac{\log \iw(G)}{n} &\le    \frac{r-1}{r} + O\left(d^{-1/(r-1)}\right).
	\end{align}
\end{thm}

\subsection{Constructions and conjectures}
\label{secTightExamples}

Part of the appeal of Theorems~\ref{thmKahn} and~\ref{thmMatch} is that the bound is exact, with  explicit extremal examples.  While we do not have an exactly matching upper bound, we provide a construction below establishing asymptotic tightness of Theorem~\ref{thm:3unif}. 

\medskip

\noindent
{\bf A construction for Theorem~\ref{thm:3unif}}. Consider the tripartite hypergraph $K$ with $3d^2$ vertices and parts $A=\{a_1, a_2,\cdots, a_{d^2}\}$, \ $B=\{b_1, b_2,\cdots, b_{d^2}\},\ C=\{c_1, c_2,\cdots, c_{d^2}\}$ and hyperedges 
$$\{a_{kd+i},b_{kd+j},c_{id+j}\}\,,$$
for all $0\leq k\leq d-1,\ 1\leq i\leq d.$

\begin{figure}
	\centering
	\includegraphics[width=8cm]{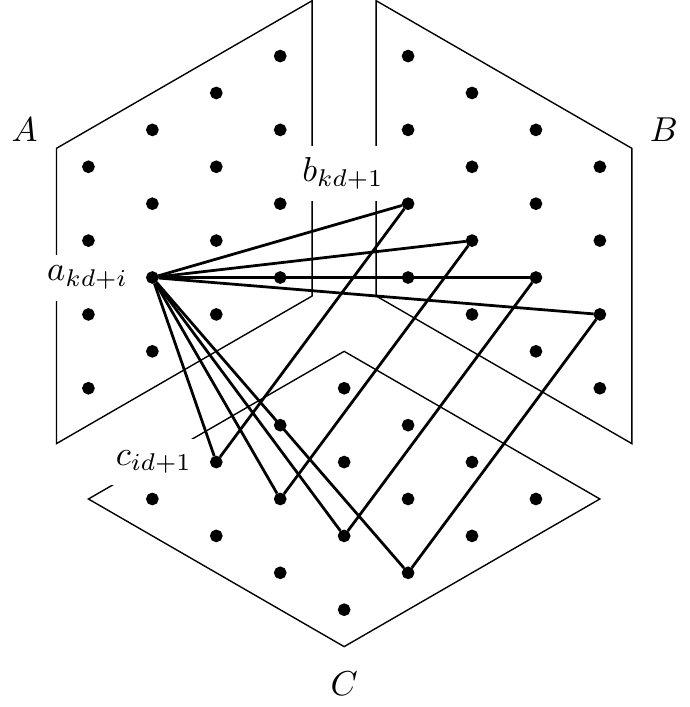}
	\caption{The neighborhood of vertex $a_{kd+i}$ in the hypergraph $K$.}
	\label{fig:cross-edge-tightness}
\end{figure}

By definition it is $3$-uniform and $d$-regular. It is also linear, since the choice of two ``adjacent" vertices (meaning those that belong to a common hyperedge) identifies $k,i,j$ and hence a unique third vertex of the edge. It is also  cross-edge free, since the graph induced by the neighborhood of a vertex is a matching. (See Figure~\ref{fig:cross-edge-tightness}.)

Now to bound the number of (strong) independent sets from below, observe that 
the graph induced by each pair of parts is a disjoint union of $d$ copies of $K_{d,d}$. Thus,
$$i_2(K)\geq (2^{d+1}-1)^{d}\geq 2^{d(d+1/2)}\,,$$
which yields
$$\frac{\log i_2(K)}{3d^2}\geq \frac{d^2+d/2}{3d^2}=\frac{1}{3}+O\left(\frac{1}{d}\right).$$

\

\medskip

We will now present constructions and give bounds that we believe are asymptotically tight in the general setting, without the (seemingly artificial) cross-edge free assumption. We will use the following easy observation.

\medskip
\noindent
{\bf Small $r$-partite hypergraphs}.

Suppose $H$ is an $r$-uniform, $d$-regular $r$-partite hypergraph on $N$ vertices.  Then $\is(H) \ge r \cdot 2^{N/r} -r +1$, and 
\begin{align}
\frac{\log \is(H)}{N} &\ge  \frac{1}{r}  + \frac{\log r}{N} -O(2^{-N/r}).
\end{align}

Thus for $N$ of order $rd$ we can get asymptotically tight graphs for Conjecture \ref{conj:strong} if we manage to construct a linear $H$. The following examples deal with some cases of $r$ and $d,$ but we are not aware of a general construction.

\medskip

\noindent
{\bf The mod graph}.

For $r=3$, let $V_1, V_2, V_3$ be vertex sets of size $d$ each identified with the integers $1, \dots, d$ and let $V(\hmod) = V_1 \cup V_2 \cup V_3$.  A triple $(v_1, v_2, v_3)$ with $v_i \in V_i$ is an edge iff $v_1 + v_2 +v_3 = 0 \ (\mbox{mod }\  d)$. This graph is $3$-uniform, $d$-regular and linear (for each $v_1 \in V_1, v_2 \in V_2$ there is a unique $v_3 \in V_3$ so that $(v_1,v_2, v_3) \in E(\hmod)$.  It is \emph{not} cross-edge free.   

We have 
\begin{align}
\is(\hmod) &= 3 \cdot 2^d -2,
\end{align}
and so
\begin{align}
\frac{\log \is(\hmod)}{3d} &= \frac{1}{3d} \log \left( 3 \cdot 2^d -2    \right )  \\
&= \frac{1}{3} + \frac{\log 3}{3d} +O(2^{-d}).  
\end{align}

Similarly, we have
\begin{align}
\iw(\hmod) &\geq 3\cdot 2^{2d} -3 \cdot 2^d +1,
\end{align}
and so
\begin{align}
\frac{\log \iw(\hmod)}{3d} &\geq \frac{1}{3d} \log \left( 3\cdot 2^{2d} -3 \cdot 2^d +1 \right )  \\
&= \frac{2}{3} + \frac{\log 3}{3d} +O(2^{-d}).  
\end{align}

For general $r$ the mod construction is not linear.     

\medskip

\noindent
{\bf A $4$-uniform construction}.

Here we describe a special case of a general construction detailed below that was brought to our attention by Dmitry Shabanov.

For $r=4$ and $d$ odd, consider the following graph.  Let $V(H_4) = V_1 \cup \cdots \cup V_4$, with each $v_i$ identified with $\{0, \dots d-1 \}$. Say $(v_1, v_2, v_3, v_4) \in E(H_4)$ if $v_1 + v_2 = v_3 \ (\mbox{mod }\  d)$ and $v_1 + 2 v_2 = v_4 \ (\mbox{mod }\  d)$. $H_4$ has the property that any two vertices in two different parts appear in exactly one edge together.  Thus $H_4$ is linear and $d$-regular.  We have $\is(H_4) = 4 \cdot 2^d -3$ and $\iw(H_4) = 4\cdot 2^{3d} - 6\cdot 2^{2d} +4 \cdot 2^d-1$, giving
\begin{align}
\frac{\log \is(H_4) }{4d} &= \frac{1}{4} + \frac{\log 4}{4d} (1+o(1)) \\
\frac{\log \iw(H_4) }{4d} &= \frac{3}{4} + \frac{\log 4}{4d} (1+o(1)).
\end{align}

\medskip
To find examples that give the asymptotics of Conjecture \ref{conj:strong} for general $r$, we have to limit our options on the degree $d:$
\medskip

\noindent
{\bf An $r$-partite, $r$-uniform, linear hypergraph, for $r \ge 3$ and prime $d > r$}. This classical construction, whose special case $r=4$ was described above and works for all odd numbers $d$. 
For $r\ge 3$, consider the $r$-uniform, $r$-partite hypergraph $H(r,d)$ with the vertex set  $V = \cup_{i=1}^r V_i$, consisting of disjoint sets $V_i = \{1,2,\dots,d\}$, for $i =1, 2, \ldots, r$. The number of vertices is $n=rd$. An $r$-tuple $(x_1,x_2,\ldots, x_r)$, with $x_i \in V_i$, is an edge iff the following set of congruences is satisfied:
\[x_3 \equiv x_1+x_2 \ (\mbox{mod }\  d), \ 
x_4 \equiv x_1+2x_2 \ (\mbox{mod }\ d), \ldots , 
x_r \equiv x_1+(r-2)x_2 \ (\mbox{mod } \ d)\,.\]

Note that when $d$ is prime, the choice of any two of the $x_i$ from an edge determine the rest, implying that $H(r,d)$ is a linear hypergraph. The number of strong independent sets equals $r \times 2^d - (r-1)$, coming from choosing any subset of a particular class $V_i$ of vertices. The number of weak independent sets is at least $r2^{(r-1)d}-\binom{r}{2}2^{(r-2)d}.$ Thus
\begin{align*}
\frac{\log \is(H(r,d)) }{rd} &= \frac{1}{r} + \frac{\log r}{rd} - O(2^{-d}) \\
\frac{\log \iw(H(r,d)) }{rd} &= \frac{r-1}{r} + \frac{\log r}{rd} - O(2^{-d}).
\end{align*}

Considering the above constructions, we conjecture that the constant in front of  the second order term in the bound on the normalized logarithm of $\is(H)$ and $\iw(H)$ should be $1$.
\begin{conj}
	\label{conj:linear}
	Suppose $r,d \ge 2$.  Then for any $r$-uniform, $d$-regular linear hypergraph $H$ on $n$ vertices
	\begin{align}
	\frac{\log \is(H)}{n} &\le \frac{1}{r} + \frac{\log r}{rd}  
	\intertext{and}
	\frac{\log \iw(H)}{n} &\le \frac{r-1}{r} + \frac{\log r}{rd}.
	\end{align}
\end{conj}
This is consistent with the above constructions. The bound for $r=2$ is attained by $K_{d,d}$, for $r=3$ by the mod graph and for $r\geq 4$ by our last example.

\section{Occupancy fraction and the hard-core model}
\label{secHardCore}

In this section we give a brief overview of the method we will use to prove Theorems~\ref{thm:3unif} and~\ref{thm:weak}.

Let $G$ be an $r$-uniform hypergraph on $n$ vertices.  Fix some $k \in \{2, \dots r\}$.   The \emph{hard-core} model on $G$ at fugacity $\lam$ is a random independent set $\mathbf I \in \mathcal I_k(G)$ chosen with probability 
\begin{align}
\label{eqhardcore}
{\Pr}_{\lam} (\mathbf I = I) &=  \frac{ \lam^{|I|}} {Z_G(\lam)} 
\end{align}
where 
\begin{align}
Z_G(\lam) &= \sum_{I \in \mathcal I_k(G)} \lam^{|I|}  \,.
\end{align}
We omit the dependence on $k$ in the notation, as it should be clear from context.  As mentioned above, the function $Z_G(\lam)$ is the independence polynomial, or  the partition function of the hard-core model: the normalizing constant that ensures that~\eqref{eqhardcore} defines a valid probability distribution. The partition function encodes a large amount of information about the independent sets of $G$: for example, $Z_G(1) = i_k(G)$, and the highest order term of the polynomial tells us both the size and number of maximum independent sets of $G$. 

The main technique used in this paper for obtaining upper bounds on the number of independent sets follows the \emph{occupancy method} of Davies, Jenssen, Perkins, and Roberts~\cite{davies2015}. We define the \emph{occupancy fraction} of the hard-core model on $G$ as
\begin{align}
\overline \alpha_G(\lam) &= \frac{1}{|V(G)|} \E_{\lam} |\mathbf I| \, ,
\end{align}
that is, the expected fraction of vertices of $G$ in a random independent set $\mathbf I$ drawn from the hard-core model. 

Crucially for our purposes, the occupancy fraction is the scaled derivative of the logarithm of the partition function:
\begin{align}
\overline \alpha_G(\lam) &= \frac{1}{|V(G)| } \frac{\sum_{I \in \mathcal I_k(G)} |I| \lam^{|I|}}{ Z_G(\lam) } \\
&= \frac{\lam} {|V(G)| }  \frac{Z_G'(\lam)     } {Z_G(\lam)} \\
&= \frac{\lam} {|V(G)| } (\ln Z_G(\lam))'  \, .
\end{align}

As $\overline \alpha_G(0) = 0$ for any $G$, we can write
\begin{align}
\label{eqOccInt}
\frac{1}{|V(G)|} \ln Z_G(\lam) &= \int_0^{\lam}  \frac{\overline \alpha_G(t)}{t} \, dt \, . 
\end{align}

In~\cite{davies2015}, Davies, Jenssen, Perkins, and Roberts proved the following theorem.
\begin{thm}
\label{thmIndOcc}
For all $d$-regular graphs $G$ and all $\lam >0$,
\begin{align}
\overline \alpha_G(\lam) &\le \overline \alpha_{K_{d,d}}(\lam) 
\end{align}
with equality if and only if $G$ is a union of copies of $K_{d,d}$.
\end{thm}
Theorem~\ref{thmIndOcc} strengthens Theorem~\ref{thmKahn} as we can  integrate the bound in Theorem~\ref{thmIndOcc} from $0$ to $1$ to obtain Theorem~\ref{thmKahn}, as in~\eqref{eqOccInt}.  Theorem~\ref{thmMatch} was proved in the same paper by proving the corresponding result for the occupancy fraction of matchings in regular graphs. To summarize, to prove an upper bound on the number of independent sets in $G$ it suffices to prove an upper bound on the occupancy fraction of the hard-core model at all fugacities $\lambda \in (0,1)$. 

To do this, we consider a collection $\set{N_v}_{v\in V(G)}$ of \emph{neighborhoods} $N_v\subset V(G)$ such that each vertex $u$ is counted in the same number $D$ of neighborhoods, i.e., $\size{\set{v: u\in N_v}} = D$ does not depend on $u$. Then we can write the occupancy fraction $\overline\alpha$ in two different ways:
\begin{align}
  \overline\alpha &= \frac{1}{n} \sum_v \Pr(v\in I) \\
  &= \frac{1}{n} \sum_u \frac{1}{D}\sum_{v: u\in N_v} \Pr(u\in I) = \frac{1}{nD} \sum_v \E(\size{I\cap N_v}).
\end{align}
In either case, we can condition on the value $J=I\setminus N$ of the independent set outside $N_v$, giving
\begin{align}
  \overline\alpha &= \frac{1}{n} \sum_v \sum_J p_J \Pr(v\in I \mid I\setminus N_v = J)\label{eqnTwoFormulations}\\
  &= \frac{1}{nD} \sum_v \sum_J p_J \E(\size{I\cap N_v} \mid I\setminus N_v = J),
\end{align}
where $p_J = \Pr(I\setminus N_v = J)$. In practice, we will group terms in the sum over $J$ by the distribution of $N_v\cap I$ given $J$, yielding a sum over a relatively small number of possible \emph{local configurations} $C_v$. Given the local configuration it is typically not difficult to compute the conditional distribution of $I\cap N_v$, but the probabilities $p_{C_v}$ are global properties depending on the graph. Instead of calculating them directly, we note that the equality of the two formulae for $\alpha$ may in fact yield nontrivial constraints on the possible values of $p_{C_v}$.  This is because when $N_v$ is actually a neighborhood of $v$ configurations where $v$ is likely to be in $I$ must in expectation have fewer neighbors of $v$ in $I$. 

To bound $\alpha$, we merely bound the sum for \emph{all} probability distributions $p_{C_v}$ subject to the equality of the two expressions for $\alpha$. This is a linear program in variables $p_{C_v}$; any feasible solution to its dual implies a bound on the objective function. In some cases a tighter bound may be obtainable by enforcing several such constraints.  Formulating such constraints and solving the resulting linear programs is the essence of the occupancy method; see~\cite{Cohen2015,girthCubicA,PottsExtremes} for several recent examples.

\section{Weak independent sets: proof of Theorem~\ref{thm:weak}}
\label{secThmWeak}

Applying the method outlined in the previous section, we prove an upper bound on the weak independent set occupancy fraction in linear, cross-edge free hypergraphs.  

\newcommand{\llambda}{\mu}

Recall that the neighborhood of $v$ is  $N(v) = \{u\in V(H)\mid u\sim v\}$. Note that since $H$ is $d$-regular, $r$-uniform, and linear, every $u\in V(H)$ is in exactly $d(r-1)$ neighborhoods. For a fixed vertex $v$, call a vertex $u\sim v$ \emph{externally covered} by $I\subseteq V(H)$ if there is some edge $f\ni u$ such that $f\cap N(v) = \{u\}$ and $f\setminus\{u\} \subseteq I$. Conditioned on $I\setminus N(v) = J$, $I$ cannot contain any neighbor $u$ of $v$ which is externally covered by $J$, so such a vertex may be safely ignored when calculating $\E(|I|\cap N(v) \mid I\setminus N(v) = J)$. We can also ignore the hyperedges containing $u$, since those constraints are automatically satisfied by ensuring that $u\not\in I$.

Given an independent set $I$ and vertex $v$, define the \emph{local configuration} $C_v$ to be the following hypergraph on the vertex set $\hat N$  consisting of $v$ and its externally uncovered neighbors: for each edge $e$ of $H$, include the edge $e\cap \hat N$ as an edge of $C_v$ if (a) $e\cap \hat N \neq \emptyset$, (b) $e\setminus \hat N \subseteq I$, and (c) $e$ does not contain any externally covered neighbor of $v$. Note that any edge containing $v$ satisfies both conditions (a) and (b), but may be omitted from $C_v$ if it contains an externally covered vertex. Conditioned on $I\setminus N(v)$, the edges of $C_v$ are the remaining constraints on $I\cap N(v)$ to ensure that $I$ is a weak independent set, and the conditional distribution of $I\cap N(v)$ is precisely the distribution given by the hard-core model on $C_v$.

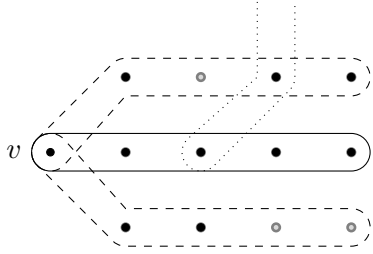
\begin{figure}
  \centering
  \begin{tikzpicture}
    \pgfmathsetmacro{\width}{.25}
    \tikzstyle{spacer}=[inner sep={\width cm*.7071}, outer sep=0cm, circle, draw=none, fill=none]
    \node[spacer, label={left:$v$}] (v) at (0,0) {};
    \node[vertex] at (v) {};
    \foreach \i in {1,2,3}{
      \foreach \j in {1,2,3,4}{
        \node[spacer] (u\i\j) at (\j, {2-\i}) {};
        \node[vertex, draw=gray, fill=lightgray, line width=1pt] at (u\i\j) {};
      }
    }
    \foreach \i in {1,3,4}{\node[vertex] at (u1\i) {};}
    \foreach \i in {1,2,3,4}{\node[vertex] at (u2\i) {};}
    \foreach \i in {1,2}{\node[vertex] at (u3\i) {};}
    \draw[dashed] (v.135) to (u11.135) arc (135:90:\width) to (u14.90) arc (90:-90:\width) to (u11.270) to (v.315) arc (315:135:\width);
    \draw (v.90) to (u24.90) arc (90:-90:\width) to (v.270) arc (270:90:\width);
    \draw[dashed] (v.215) to (u31.215) arc (215:270:\width) to (u34.270) arc (-90:90:\width) to (u31.90) to (v.45) arc (45:215:\width);
    \node[spacer] (x) at (3,2) {};
    \draw[dotted] (x.180) to (u13.180) to (u22.135) arc (135:315:\width) to (u13.315) arc (-45:0:\width) to (x.0);
  \end{tikzpicture}
  \caption{A neighborhood configuration of a vertex $v$ in a $3$-regular, $5$-uniform linear hypergraph (in the case of weak independent sets). Externally covered vertices are greyed out (their external neighbors in $I$ are not shown) and are omitted from the configuration $C_v$ along with the edges connecting them to $v$ (dashed). The dotted edge is a cross-edge, which we disallow. This configuration corresponds to parameters $j=1$ and $k=5$.}
  \label{fig:weak-neighborhood}
\end{figure}

In the case where $H$ is cross-edge free the only possible local configurations consist of $v$, $j\leq d$ edges containing $v$ each of size $r$, and $k\leq (d-j)(r-2)$ vertices whose adjacencies to $v$ have been omitted from $C_v$. The configuration is completely characterized by the parameters $j$ and $k$. (If cross-edges were allowed, $C_v$ might have additional edges not containing $v$.) Keeping the above shorthand convention $\llambda = 1+\lambda$, the partition function of such a local configuration is $Z_{j,k} = (Z_{j}^- + \lambda Z_{j}^+)\llambda^k$, where $Z_{j}^- = \llambda^{(r-1)j}$ and $Z_{j}^+ = (\llambda^{r-1}-\lambda^{r-1})^j$ count independent sets on the neighbors of $v$ in $C_v$, conditioned on $v\not\in I$ and $v\in I$, respectively.
The conditional probability that $v \in I$ given the local cofiguration is then
\begin{align}
  \alpha^v_{j} &= \frac{\lambda Z_j^+}{Z_j^- + \lambda Z_j^+}.
\end{align}
Using the formulas
\begin{align*}
\left(Z_{j}^{-}\right)^{\prime} &=\frac{(r-1)j}{\mu} Z_{j}^{-} \\
\lambda\left(Z_{j}^{+}\right)^{\prime} &=\frac{(r-1) j}{\mu}\left(\lambda Z_{j}^{+}-Z_{j}^{+} \frac{\lambda^{r-1}}{\mu^{r-1}-\lambda^{r-1}}\right),
\end{align*}
the conditional expectation of the fraction of occupied vertices among the neighbors of $v$ (not including $v$ itself) is
$$
\begin{aligned}
	\alpha_{j, k}^{N} &= \frac{1}{d(r-1)}\left(k \frac{\lambda}{\mu}+\lambda \frac{\left(Z_{j}^{-}+\lambda Z_{j}^{+}\right)^{\prime}}{Z_{j}^{-}+\lambda Z_{j}^{+}}-\frac{\lambda Z_{j}^{+}}{Z_{j}^{-}+\lambda Z_{j}^{+}}\right) \\
	&=\frac{\lambda}{d(r-1)}\left(\frac{k}{\mu}+\frac{(r-1) j}{\mu}\cdot \frac{Z_{j}^{-}+\lambda Z_{j}^{+}-Z_{j}^{+} \frac{\lambda^{r-1}}{\mu^{r-1}-\lambda^{r-1}}}{Z_{j}^{-}+\lambda Z_{j}^{+}}\right) \\
	&=\frac{\lambda}{d(r-1)}\left(\frac{k}{\mu}+\frac{(r-1) j}{\mu}\left(1-\frac{Z_{j}^{+} \frac{\lambda^{r-1}}{\mu^{r-1}-\lambda^{r-1}}}{Z_{j}^{-}+\lambda Z_{j}^{+}}\right)\right) \\
	&=\frac{\lambda}{\mu d (r-1)}\left(k+(r-1)j-(r-1)j \frac{\alpha_{j}^{v}}{\lambda}\left(\frac{\lambda^{r-1}}{\mu^{r-1}-\lambda^{r-1}}\right)\right)
\end{aligned}
$$

Following \eqref{eqnTwoFormulations}, we can write
\begin{align}
\label{eqOvAlpha}
  \overline\alpha &= \sum_{j,k} p_{j,k} \alpha_j^v = \sum_{j,k} p_{j,k} \alpha_{j,k}^N,
\end{align}
where $p_{j,k}$ is the probability that $C_v$ is the local configuration with parameters $j,k$ when we pick $v$ uniformly at random and $I$ from the hard-core model with fugacity $\lambda$.  That is, we take expectations over the conditional expectations given by the preceding formulas.

To apply the occupancy method, we relax the optimization problem of maximizing $\overline \alpha_G(\lam)$ over all graphs to maximizing $\overline \alpha$ as given by~\eqref{eqOvAlpha} over all probability distributions $p_{j,k}$ subject to the constraint that these two formulations in~\eqref{eqOvAlpha} are equal. This yields the primal LP
\begin{align}
  \alpha^*(\lam) =\max \sum_{j,k} p_{j,k} \alpha_j^v &\qquad \text{s.t.}\\
  \sum_{j,k} p_{j,k} (\alpha_j^v - \alpha_{j,k}^N) &= 0\\
  \sum_{j,k} p_{j,k} &= 1\\
  p_{j,k} &\geq 0 \,.
\end{align}

Following the discussion in Section~\ref{secHardCore}, we have that for any $d$-regular, linear, cross-edge free hypergraph $G$,
\[ \overline \alpha_G(\lam) \le \alpha^*(\lam) \,,\]  
and therefore,
\begin{equation}
\label{eqlogWeq}
\frac{\log \iw(G)}{n} \le  \frac{1}{\ln 2}  \int_0^{\lam}  \frac{ \alpha^*(t)}{t} \, dt \,.
\end{equation}

In what follows we derive an upper bound on $\alpha^*(\lam)$ and integrate this bound to obtain Theorem~\ref{thm:weak}.  

\begin{prop}\label{propOccupancyWeak}
 For any $\lam \in [0,1]$, 
\begin{align}
  \frac{ \alpha^*(\lam)}{\lam} 
  &= \frac{(r-1) \parens*{\llambda \lambda^r - \lambda\llambda^r} \llambda^{d(r-1)} + \parens*{\llambda^2 \lambda^r + \lambda (r-\llambda-1) \llambda^r} \parens*{\llambda^{r-1} - \lambda^{r-1}}^d}%
          {\llambda \parens*{r \parens*{\llambda \lambda^r - \lambda \llambda^r} \llambda^{d(r-1)} + \lambda \parens*{\llambda \lambda^r + (r-\llambda) \llambda^r} \parens*{\llambda^{r-1} - \lambda^{r-1}}^d}} \\
          & =: \frac{\alpha_{\mathrm w}(r,d,\lam)}{\lam}\,,
\end{align}
where $\alpha^*(\lam)$ is defined in the linear program above, and where we have written $\llambda \defeq \lambda + 1$ for brevity.
\end{prop}

Before proving Proposition~\ref{propOccupancyWeak}, we derive Theorem~\ref{thm:weak} from it.  
 Along with~\eqref{eqlogWeq}, the following claim  gives Theorem~\ref{thm:weak}.

\begin{claim}
\label{claimWeakAsymp}
For any fixed $r \ge 3$,
\begin{align}
\frac{1}{\ln 2}\int_0^1 \frac{ \alpha_{\mathrm{w}}(r,d,\lam)}{\lam} \, d \lam = \frac{r-1}{r} + \frac{c_{\mathrm w,r}}{d^{1/(r-1)}} (1+o_d(1))
\end{align}
where 
\begin{align}
c_{\mathrm w,r}&= \frac{1}{\ln 2} \int_0^\infty   \frac{1}{r + r^2 \left( e^{c^{r-1}}-1  \right) } \, d x  .
\end{align}
\end{claim}
\begin{proof}
Let $G(r,d,\lam) = \frac{ \alpha_{\mathrm{w}}(r,d,\lam)}{\lam} - \frac{r-1}{r\llambda}$.  That is,
\begin{align}
  G(r,d,\lam) &=
  \frac{\parens*{\lambda^r (\lambda+r) - \lambda \llambda^r} \parens*{\llambda^{r-1} - \lambda^{r-1}}^d}%
       {r \parens*{r \parens*{\llambda \lambda^r - \lambda \llambda^r} \llambda^{d(r-1)} + \lambda\parens*{\llambda \lambda^r + (r-\llambda) \llambda^r} \parens*{\llambda^{r-1} - \lambda^{r-1}}^d}}\,.
\end{align}

   We can take the asymptotics of $G(r,d,\lam)$ as $d \to \infty$ for $\lam = \Theta(d^{-1/(r-1)})$:
\begin{align}
  G(r,d,\lam) &= 
 (1+ o_d(1) )  \frac{-\lambda  \llambda^r \parens*{\llambda^{r-1}-\lambda^{r-1}}^d}%
       {r \parens*{\lambda  \parens*{\llambda \lambda^r + (r-\llambda) \llambda^r} \parens*{\llambda^{r-1} - \lambda^{r-1}}^d - \lambda r \llambda^{d(r-1) + r}}} \\
  &=(1+ o_d(1) )  \frac{1}{r \parens*{r \parens*{\llambda^{d(r-1)} \parens*{\llambda^{r-1} - \lambda^{r-1}}^{-d} - 1} +\llambda - \lambda^r \llambda^{1-r}}} \\
  &= (1+ o_d(1) ) \frac{1}{r^2 \parens*{\llambda^{d(r-1)} \parens*{\llambda^{r-1} - \lambda^{r-1}}^{-d} - 1} + r} \\
  &=(1+ o_d(1) )  \frac{1}{r+r^2 \parens*{\frac{\llambda^{d(r-1)}} {\parens*{\llambda^{r-1} - \lambda^{r-1}}^d} - 1}} \,,
\end{align}
and with the parameterization $\lam =  c  d^{-1/(r-1)}$, we obtain
\begin{align}
G(r,d,\lam) &=  \frac{1}{r + r^2 \left( e^{c^{r-1}}-1  \right) }  (1+ o_d(1) ) \, ,
\end{align}
and note that
{\small
\begin{align}
\frac{1}{\ln 2}\int_0^1 \frac{ \alpha_{\mathrm{w}}(r,d,\lam)}{\lam} \, d \lam  &= \frac{1}{\ln 2}\int_{0}^1 \frac{r-1}{r} \frac{1}{1+\lam} \, d \lam + (1+o_d(1)) \frac{1}{\ln 2} \int_0^{\infty}   \frac{1}{r + r^2 \left( e^{c^{r-1}}-1  \right) } \, dc \\
& = \frac{r-1}{r} + \frac{c_{\mathrm w,r}}{d^{1/(r-1)}}  (1+o_d(1)) 
\end{align}
}
 as desired.

\end{proof}

\begin{proof}[Proof of Proposition~\ref{propOccupancyWeak}]

We can construct a candidate optimal solution $\tilde p_{j,k}$ by putting support only on $(j,k) \in \{(d,0), (0,(r-2)d)\}$.  Solving the constraint yields 
\begin{align}
\tilde p_{d,0} 
&= \frac{\alpha_0^v - \alpha_{0,d(r-2)}^N}{\alpha_0^v - \alpha_{0,d(r-2)}^N - \alpha_d^v + \alpha_{d,0}^N}\\
&= \frac{\lambda}{\lambda + (r-1)\llambda (\alpha_{d,0}^N - \alpha_d^v)}
\end{align}
and
\begin{align}
  \frac{\tilde\alpha}{\lambda} 
  &= \frac{(r-1)\alpha_{d,0}^N - (r-2)\alpha_d^v}{\lambda + (r-1)\llambda(\alpha_{d,0}^N - \alpha_d^v)}\\
  &= \frac{1}{\llambda} \parens*{1 - \tilde p_{d,0}\frac{Z_d^- - Z_d^+}{Z_d^- + \lambda Z_d^+}}\\
  &= \frac{1}{\llambda}\parens*{1 - \frac{Z_d^- - Z_d^+}{r(Z_d^- + \lambda Z_d^+) - (r-1)\pfrac{\llambda^r - \lambda^r}{\llambda^{r-1} - \lambda^{r-1}} Z_d^+}}.
\end{align}
As a check, note that $\alpha_0^v > \alpha_{0,d(r-2)}^N \geq 0$ and $\alpha_{d,0}^N \geq \alpha_d^v \geq 0$, so $0\leq \tilde p_{d,0} \leq 1$.

The dual LP (in variables $\Lambda_c, \Lambda_p$ corresponding to the two equality constraints of the primal) is
\begin{align}
  \min \Lambda_p &\qquad \text{s.t.}\\
  \Lambda_p + (\alpha_j^v - \alpha_{j,k}^N)\Lambda_c &\geq \alpha_j^v \qquad\forall\ j,k.
\end{align}

Guided by our candidate primal optimal solution, we can find candidate dual variables by solving for equality in the $(j,k) = (d,0), (0,(r-2)d)$ constraints.  This yields $\tilde \Lam_p = \tilde \alpha$ and
\begin{align}
  \tilde \Lambda_c &= (r-1)\parens*{1 - \tfrac{\llambda}{\lambda} \tilde\Lambda_p}.
\end{align}

What remains to show is that this candidate dual solution $\tilde \Lam_p, \tilde \Lam_c$ is feasible; that is, 
\begin{align}
 \tilde \Lambda_p + (\alpha_j^v - \alpha_{j,k}^N)\tilde \Lambda_c &\geq \alpha_j^v \qquad\forall\ j,k  \,.
\end{align}
Note that $\tilde\Lambda_p = \tilde\alpha \leq \frac{\lambda}{\llambda}$ (as $\alpha_j^v \leq \frac{\lambda}{\llambda}$ for all $j$) and so $\tilde\Lambda_c \geq 0$. Thus we may assume that $k = (d-j)(r-2)$ as decreasing $k$  only decreases  $\alpha_{j,k}^N$ and makes the constraint easier to satisfy.  So we must show that for $j=0, \dots d$, and all $\lambda \in [0,1]$,
\begin{equation}
	\label{poly_ineq}
	\tilde \Lambda_p + (\alpha_j^v - \alpha_{j,(d-j)(r-2)}^N)\tilde \Lambda_c \geq \alpha_j^v   \, .
\end{equation}

Equivalently, writing $\alpha_{j}:=\alpha_{j}^{v} / \lambda$ and $s$ for $r-1$,
\begin{align}
	\frac{\tilde\Lambda_{p}}{\lambda} & \geq \frac{s \alpha_{j, k}^{N} / \lambda-(s-1) \alpha_{j}^{v} / \lambda}{1+s \mu\left(\alpha_{j, k}^{N} / \lambda-\alpha_{j}^{v} / \lambda\right)} \\
	&=\frac{\frac{1}{\mu d}\left(k+s j-s j \alpha_{j}\left(\frac{\lambda^{s}}{\mu^{s}-\lambda^{s}}\right)\right)-(s-1) \alpha_{j}}{1+\frac{1}{d}\left(k+s j-s j \alpha_{j}\left(\frac{\lambda^{s}}{\mu^{s}-\lambda^{s}}\right)\right)-s \mu \alpha_{j}} \\
	&=\left(\frac{1}{\mu}\right) \frac{k+s j-s j \alpha_{j}\left(\frac{\lambda^{s}}{\mu^{s}-\lambda^{s}}\right)-(s-1) \mu d \alpha_{j}}{d+k+s j-s j \alpha_{j}\left(\frac{\lambda^{s}}{\mu^{s}-\lambda^{s}}\right)-s \mu d \alpha_{j}} \\
	&=\frac{1}{\mu}\left(1+\frac{\mu d \alpha_{j}-d}{s d+j-s j \alpha_{j}\left(\frac{\lambda^{s}}{\mu^{s}-\lambda^{s}}\right)-s \mu d \alpha_{j}}\right) \\
	&=\frac{1}{\mu}\left(1-\frac{d\left(\mu^{s}-\lambda^{s}\right)}{\frac{j\left(\lambda^{s} s \alpha_{j}-\left(\mu^{s}-\lambda^{s}\right)\right)}{\mu \alpha_{j}-1}+s d\left(\mu^{s}-\lambda^{s}\right)}\right)=:f(j)
\end{align}
Note that the left hand side is equal to $f(d)$, so we need to show that $f(j) \leq f(d)$. This is equivalent to showing that $g(j) \leq g(d)$ for
\begin{align}
	g(j) &=d\left(\mu^{s}-\lambda^{s}\right)\left(\frac{1}{1-\mu f(j)}-s\right) \\
	&=\frac{j\left(s \alpha_{j} \lambda^{s}-\left(\mu^{s}-\lambda^{s}\right)\right)}{\mu \alpha_{j}-1} \\
	&=\frac{j\left(s Z_{j}^{+} \lambda^{s}-\left(\mu^{s}-\lambda^{s}\right)\left(Z_{j}^{-}+\lambda Z_{j}^{+}\right)\right)}{\mu Z_{j}^{+}-\left(Z_{j}^{-}+\lambda Z_{j}^{+}\right)} \\
	&=\frac{j\left(s Z_{j}^{+} \lambda^{s}+Z_{j}^{-} \lambda^{s}-\left(Z_{j+1}^{-}+\lambda Z_{j+1}^{+}\right)\right)}{Z_{j}^{+}-Z_{j}^{-}} \\
	&=j \frac{C_{j}}{B_{j}}
\end{align}
where
$B_j:=\mu^{js}-\nu^j$, 
$C_j:=\mu^{js}\nu+\lambda\gamma\nu^j$, and $\mu,\nu,\gamma$ are the auxiliary variables
\begin{align*}
	\mu &:=1+\lambda \\
	\nu &:=(1+\lambda)^{s}-\lambda^{s} \\
	\gamma&:=(1+\lambda)^{s}-\lambda^{s-1}(\lambda+s).
\end{align*} 
Now $g(d)\geq g(j)$ can be rewritten as $dB_jC_d\geq jB_dC_j,$ which in turn is equivalent to
\begin{align*}
	&(d-j)\mu^{(d+j)s}-(d-j)\lambda\gamma \nu^{d+j-1}\geq\\
	&d\mu^{sd}\nu^j-j\mu^{js}\nu^d+j\lambda\gamma\mu^{sd}\nu^{j-1}-d\lambda\gamma\mu^{js}\nu^{d-1}
\end{align*}
or 
\begin{align}
	(d-j)\mu^{(d+j)s}\geq & \mu^{js}\nu^{j-1}\Big(\nu(d\mu^{(d-j)s}-j\nu^{d-j})-\lambda\gamma(d\nu^{d-j}-j\mu^{(d-j)s})\Big) \\
	&+(d-j)\lambda\gamma\nu^{d+j-1} \label{prop11_1} 
\end{align}
Notice that it suffices to show $g$ is increasing. Thus, we may assume $d=j+1$ and reduce \eqref{prop11_1} to
$$\mu^{(2j+1)s}\geq \mu^{js}\nu^{j-1}\Big(\nu\big((j+1)\mu^{s}-j\nu\big)-\lambda\gamma\big((j+1)\nu-j\mu^{s}\big)\Big)+\lambda\gamma\nu^{2j}.$$
Set $\rho:=\mu^{s}$ and $\sigma:=\lambda^{s}$, thus $\rho-\nu=\sigma$.  The desired inequality can be now written as
\begin{equation}\label{prop11_2}
	\rho^{2j+1}\geq \rho^{j}\nu^{j-1}\Big(\nu\big((j+1)\rho-j\nu\big)-\lambda\gamma\big((j+1)\nu-j\rho\big)\Big)+\lambda\gamma\nu^{2j}.
\end{equation}
Substituting
$$(j+1)\rho-j\nu= (j+1)\rho-j(\rho-\sigma)=\rho+j\sigma$$
and
$$(j+1)\nu-j\rho=(j+1)(\rho-\sigma)-j\rho=\rho-(j+1)\sigma$$
into \eqref{prop11_2} we get the equivalent
\begin{equation}\label{prop11_3}
	\rho^{2j+1}\geq \rho^{j}\nu^{j-1}\Big(\nu\big(\rho+j\sigma\big)-\lambda\gamma\big(\rho-(j+1)\sigma\big)\Big)+\lambda\gamma\nu^{2j}.
\end{equation}
We will prove \eqref{prop11_3} by induction on $j$.
For $j=1$, we substitute $\nu=\rho-\sigma$ and, after expanding and canceling out, we are left with
$$\rho\sigma^2\geq \lambda\gamma\sigma^{2},$$
which holds, since $\rho\geq \gamma\geq\lambda\gamma$.
Now assume \eqref{prop11_3} holds for some $j$. Using it for $\rho^{2j+3}=\rho^{2}\cdot\rho^{2j+1}$, the statement for $j+1$ reduces to
\begin{align*}
	&\rho^{2}\Bigg(\rho^{j}\nu^{j-1}\Big(\nu\big(\rho+j\sigma\big)-\lambda\gamma\big(\rho-(j+1)\sigma\big)\Big)+\lambda\nu^{2j}\gamma\Bigg)\geq\\
	&\rho^{j+1}\nu^{j}\Big(\nu\big(\rho+(j+1)\sigma\big)-\lambda\gamma\big(\rho-(j+2)\sigma\big)\Big)+\lambda\nu^{2j+2}\gamma.
\end{align*}
By writing $(k+1)\sigma=k\sigma+\sigma$ for $k=j,j+1$ and moving everything to the left hand side we get	
$$\rho^{j+1}\nu^{j-1}\Bigg(\nu\big(\rho+j\sigma\big)(\rho-\nu)-\lambda\gamma\big(\rho-(j+1)\sigma\big)(\rho-\nu)-\nu\sigma(\nu+\lambda\gamma)\Bigg)+\lambda\nu^{2j}\gamma(\rho^{2}-\nu^2)\geq0.$$
Since $\rho-\nu=\sigma,$ we can cancel $\sigma\nu^{j-1}$, and we are left with proving
$$
\rho^{j+1}\Bigg(\nu\big(\rho+j\sigma\big)-\lambda\gamma\big(\rho-(j+1)\sigma\big)-\nu(\nu+\lambda\gamma)\Bigg)+\lambda\nu^{j+1}\gamma(\rho+\nu)\geq0,
$$
which is equivalent to	
$$\rho^{j+1}\nu(\rho-\nu)+j\sigma\nu\rho^{j+1}+(j+1)\lambda\gamma\sigma\rho^{j+1}\geq\lambda\gamma\rho(\rho^{j+1}-\nu^{j+1})+\lambda\gamma\nu(\rho^{j+1}-\nu^{j+1}),$$
or
$$(j+1)\rho^{j+1}\sigma(\nu+\lambda\gamma)\geq \lambda\gamma(\rho^{j+1}-\nu^{j+1})(\rho+\nu).$$
Now observe that
\begin{align*}
	\rho^{j+1}-\nu^{j+1}&=\sigma\sum_{k=0}^{j}\rho^{k}\nu^{j-k} \leq (j+1)\sigma\rho^{j}
\end{align*}
So it suffices to prove
$$\rho(\nu+\lambda\gamma)\geq\lambda\gamma(\rho+\nu),$$
i.e. $\rho\geq \lambda\gamma,$ which is obviously true. 
\end{proof}

\section{Strong independent sets in $3$-uniform hypergraphs}
\label{secStronIS}

Now we turn to strong independent sets and prove Theorem~\ref{thm:3unif}.    Using another linear programming relaxation, we will prove the following upper bound on the occupancy fraction.  

\begin{prop}\label{thmoccupancy3}
Let
\begin{align}
\label{eqAlpaFormula}
\frac{\alpha_{\mathrm s}(r,d,\lambda)}{\lambda} &= \frac{\sum_{t=1}^{r} \binom{r-1}{t-1}\prod_{i=t}^{r-1}((1+i\lambda)^{d-1}-1)}{\sum_{t=1}^{r} ((1+(t-1)\lambda)^d + \lambda)\binom{r-1}{t-1}\prod_{i=t}^{r-1}((1+i\lambda)^{d-1}-1)} \,.
\end{align}
Then for any $d$-regular, $3$-uniform, linear, cross-edge free hypergraph $G$, and for any $\lambda > 0$,
\[ \overline \alpha_G(\lambda) \leq \alpha_{\mathrm s}(3,d,\lambda) \, .\]
\end{prop}
Note that we define $\alpha_{\mathrm s}(r,d,\lambda)$ for general $r$, while only considering $r=3$ in the proposition.  We believe the inequality holds for $r\le 6$, and leave proving this as an open problem.  

Once we  prove Proposition~\ref{thmoccupancy3},  Theorem~\ref{thm:3unif} will follow via integration from the next claim. 
\begin{claim}
\begin{align}
\frac{1}{\ln 2} \int_0^1 \frac{\alpha_{\mathrm s}(3,d,\lam)}{\lam} \, d \lam = \frac{1}{3} + \frac{c_{\mathrm s, 3}}{d} (1+o_d(1))
\end{align}
where 
\[c_{\mathrm s, 3}= \frac{1}{\ln 2} \int_0^ \infty \frac{ 3e^{2c}-1  }{ 3(1-3e^c + 3 e^{3c}    )  }  \, dc   \approx 0.603772 \, .\]
\end{claim}
\begin{proof}
The proof is similar to that of Claim~\ref{claimWeakAsymp}.  Let 
{\small
\begin{align}
G_3(d,\lam) &=  \frac{\alpha_{\mathrm s}(3,d,\lam)}{\lam} - \frac{1}{3} \frac{1}{1+\lam} \\
&= \frac{ 3 (1 + 2 \lam)^d-1 - 2 \lam }{3 (1 + \lam) \left [1 - 3 (1 + \lam)^d + 3 (1 + \lam)^d (1 + 2 \lam)^d + 
   \lam (2 - 6 (1 + \lam)^d + 3 (1 + 2 \lam)^d) \right  ]} \, .
\end{align}
}

Paramterizing by $\lam = c/d$, we have
\begin{align}
G_3(d,c/d)  &=  \frac{\alpha_{\mathrm s}(3,d,\lam)}{\lam} - \frac{1}{3} \frac{1}{1+\lam} \\
&= \frac{ 3e^{2c} -1 +O(1/d)  }{3 (1 + c/d) \left [1 - 3 e^c +3e^{3c}  +O(1/d) 
   \right  ]} \\
&= \frac{ 1-3e^{2c}  }{ 3(3e^c - 3 e^{3c} -1   )  }(1 +O(1/d))
\end{align}
as $d \to \infty$,  and so we have 
\begin{align}
\frac{1}{\ln 2} \int_0^1 \frac{\alpha_{\mathrm s}(3,d,\lam)}{\lam}\, d \lam &= \frac{1}{\ln 2}\int_0^1 \frac{1}{3} \frac{1}{1+\lam} + G_3(d,\lam) \, d\lam \\
&=  \frac{1}{\ln 2}\int_0^1 \frac{1}{3} \frac{1}{1+\lam} \, d\lam +  \frac{1}{\ln 2} \int_0^d  \frac{G_3(d,\lam)}{d}   \, dc \\
&= \frac{1}{3} + \frac{c_{\mathrm s, 3}}{d}( 1 + O(1/d)) \, .
\end{align}

\end{proof}

Now we prove Proposition~\ref{thmoccupancy3}.

We can get a local estimate of $\alpha_G$ by examining (along with the independent set $I$) a uniformly random vertex $v$ and a random edge $e$ containing $v$, so that $\alpha_G = \Pr[v\in I]$. Note that because $G$ is regular and uniform this is equivalent to picking $e$ uniformly and then picking $v$ uniformly from $e$.

Say a vertex $x$ is \emph{covered} by a vertex $y$ if $y\in I$ and $x\sim y$. Note that any $x\in I$ is uncovered. Call an uncovered vertex which is also unoccupied \emph{available}, and let $A$ be the set of available vertices. Let $N(v)$ denote the neighborhood of $v$, and let $\hat N(v) = N(v) \cup \{v\}$.

Call a vertex \emph{externally uncovered} if it is not covered by any vertex outside of $\hat N(v)$, and let $C_v$ be the hypergraph $G$ restricted to $v$ and its externally uncovered neighbors (keeping all partial edges through $v$, so that $v$ has still degree $d$, but $C_v$ is no longer uniform). Let $\C$ be the collection of all such possible configurations. For each $C\in \C$ write $p(C) = \Pr[C_v = C]$ for the distribution of $C_v$ and let $P_C(\lambda)$ be the partition function for the hypergraph $C$. Note that this partition function includes one configuration with $v\in I$ (of weight $\lambda$). 

We are interested in maximizing
\[\alpha_G = \sum_{C\in \C} p(C) \Pr[v\in I \mid C_v = C] = \sum_{C\in \C} p(C) \frac{\lambda}{P_C(\lambda)}\]
over all hypergraphs $G$. However, the only terms in this formula which depend on the original hypergraph $G$ at all are the probabilities $p(C)$. Thus it will be useful to know more about which distributions $p$ can actually arise from hypergraphs in this way.

Let $t(e) \coloneqq \size{e\cap A}$ be the number of available vertices in $e$. We also know that
\begin{align}
    \Pr[v \in A \mid t(e) = t] = t/r
\end{align}
for each $0\leq t\leq r$. Conditioning on $C_v = C$, we have
\begin{align}
    \frac{t}{r}\,\Pr[t(e) = t]
    &= \Pr[v\in A,\ t(e) = t]\\
    &= \sum_C p(C) \Pr[v\in A,\ t(e) = t \mid C_v = C]\\
    &= \sum_C p(C) \Pr[v\in A \mid C_v = C] \Pr[t(e) = t \mid v\in A,\ C_v = C].
\end{align}
We can calculate $\Pr[v\in A \mid C_v = C] = 1/P_C(\lambda)$ (since only the empty independent set on $C$ leaves $v$ available) and 
\[\Pr[t(e) = t \mid v\in A,\ C_v = C] = \frac{d_t(C)}{d} \eqdef \eta_t(C),\]
where $d_t(C)$ is the number of size-$t$ edges containing $v$ in $C$ (since whenever $v$ is available $t(e) = \size{e}$ and all $d$ edges containing $v$ are equally likely). Thus the probabilities $p(C)$ must satisfy
\begin{align}
    \sum_C p(C) \frac{\eta_t(C)}{P_C(\lambda)} 
    &= \frac{t}{r} \Pr[t(e) = t]
    = \frac{t}{r} \sum_C p(C) \Pr[t(e) = t \mid C_v = C]
\end{align}
giving linear constraints
\begin{align}
    \sum_C p(C) \left(t\,\Pr[t(e) = t \mid C_v = C] - \frac{r\,\eta_t(C)}{P_C(\lambda)} \right) = 0 \qquad \forall\, 0\leq t \leq r.
\end{align}
When $t=r$ the constraint holds for any choice of $p(C)$, since $t(e) = r$ precisely when $v$ is available and we pick an edge $e$ of size $r$ in $C$. It is also trivial for $t=0$, since $\eta_t = 0$ (every edge containing $v$ has size at least $1$).

These linear constraints (along with the constraint that $p$ should be a probability distribution over neighborhood configurations) give an LP relaxation for the problem of maximizing the occupancy fraction over all $d$-regular, $r$-uniform linear hypergraphs $G$, and the optimal probability distribution will give an upper bound on the occupancy fraction of such a graph---if we can solve the LP.

It remains to calculate $\Pr[t(e) = t \mid C_v = C]$, which can be quite complicated. However, the computation is vastly simplified by assuming that the hypergraph is cross-edge free. 
The possible neighborhood configurations $C$ in a cross-edge free hypergraph are completely parameterized by the number of edges $d_t(C)$ of each size $t$, as these are the only nontrivial edges in $C_v$. 
For such a neighborhood configuration $C$,
\[P_C(\lambda) = \lambda + \prod_{s=1}^r (1 + (s-1)\lambda)^{d_s(C)}.\]
For $t\neq 0$, one can obtain $t(e) = t$ given $C$ either by picking $e$ to be an edge of size $t$ and taking the empty independent set or by picking an edge of size $t+1$ and covering $v$ by one or more vertices outside that edge (the edge itself must be unoccupied, of course). That is, for $1\leq t \leq r$
\begin{align}
    \Pr[t(e) = t \mid C_v = C] = \frac{\eta_t(C) + \eta_{t+1}(C) (P_{C_t}(\lambda) - 1)}{P_C(\lambda)},
\end{align}
where $P_{C_t}(\lambda)$ is the partition function for $C$ with an edge of size $t+1$ and $v$ removed (this is just a collection of disjoint edges). In particular,
\begin{align}
    P_{C_t}(\lambda) &= \frac{P_C(\lambda) - \lambda}{1+t\lambda}.
\end{align}

Finally, we can write a linear program relaxation of our problem with variables $p(C)$:
\begin{align}\label{eqn:primal}
  \frac{\alpha^*}{\lambda} = \max \sum_{C} p(C) \frac{1}{P_C(\lambda)} \qquad \text{subject to}\\
  \sum_C \frac{p(C)}{P_C(\lambda)}\left(\eta_t(C) + \eta_{t+1}(C) \left(\frac{P_C - \lambda}{1 + t\lambda} - 1\right) - \frac{r\,\eta_t}{t}\right) &= 0 \qquad \forall\,1\leq t\leq r-1\\
  \sum_C p(C) &= 1\\
  p(C) &\geq 0 \qquad \forall C\in \C
\end{align}

\begin{remark}
  This LP relaxation generalizes both the relaxation for independent sets and that for matchings used in \cite{davies2015}, which correspond to the cases $r=2$ and $d=2$, respectively.
\end{remark}

We will use LP duality to show that the optimizer $\popt$ of this relaxation is supported on the neighborhoods $I_t$ with with $\eta_t(I_t) = 1$ (so that all of the edges in $I_t$ have size $t$). There is in fact a unique feasible solution with this support, which is realized in the case $r=2$ by $K_{d,d}$ and in the case $d=2$ by the $r\times r$ grid $K_{r,r}^T$. For $d,r>2$ the optimal solution for the relaxation does not seem likely to be feasible for the unrelaxed problem (i.e. cannot be realized by a hypergraph) and so the relaxation probably does not provide a tight bound. 

If we enforce support only on configurations $I_s$ the only nonzero terms in the primal constraint for $t$ are those with $C \in\set{I_t, I_{t+1}}$. Writing $q(s) = \popt(I_s)/P_{I_s}$, the constraint then becomes
\[q(t)(1 - \tfrac{r}{t}) + q(t+1)((1+t\lambda)^{d-1} - 1) = 0 \qquad \forall 1 \leq t \leq r-1.\]
This is effectively a recursion in $q(t)$, along with the constraint that
\[\sum_{t=1}^r \popt(I_t) = \sum_{t=1}^r q(t) P_{I_t} = 1.\]
Writing
\begin{align}
  v(t) &\defeq \frac{q(t)}{q(r)} = \prod_{i=t}^{r-1} \frac{q(i)}{q(i+1)}\\
  &= \prod_{i=t}^{r-1} \parens[\Big]{\frac{i}{r-i}}((1+i\lambda)^{d-1} - 1) = \binom{r-1}{t-1} \prod_{i=t}^{r-1}((1+i\lambda)^{d-1} - 1)
\end{align}
and $Z = \sum_{t=1}^r P_{I_t}(\lambda) v(t)$, the proposed solution to the primal is
\begin{align}
  q(t) &= \frac{v(t)}{Z}\\
  \popt(I_t) &= P_{I_t}(\lambda) q(t) = \frac{P_{I_t}(\lambda) v(t)}{Z}.
\end{align}
The objective function evaluated at this solution is
\begin{align}\label{eqn:opt}
  \frac{\alpha^*}{\lambda} &= \sum_{t=1}^r \frac{p(I_t)}{P_{I_t}} = \sum_{t=1}^r q(t)
    = \frac{\sum_t v(t)}{Z} \, .
\end{align}

The LP dual to \eqref{eqn:primal} is (in variables $\Lambda$ and $\Lambda_t$, $1\leq t \leq r-1$)
{\small
\begin{align}\label{eqn:dual-orig}
    \frac{\alpha^*}{\lambda} = \min \Lambda \qquad \text{subject to }\qquad\qquad\qquad\\
    \Lambda P_C(\lambda) + \sum_{t=1}^{r-1} \Lambda_t \left(\eta_{t+1}(C) \left(\frac{P_C(\lambda)-\lambda}{1+t\lambda} - 1\right) - \eta_t(C) \left(\frac{r}{t}-1\right) \right) &\geq 1 \quad\forall C\in \C
\end{align}
}

To show that the primal optimum is supported on the configurations $I_s$, we show that there is a feasible solution to the dual for which the corresponding constraints are tight. We can solve for candidate values $\Lopt$ and $\Lopt_t$ by setting these $r$ constraints to equality. To simplify notation, we will write 
\[Q_t \coloneqq P_{I_t}(\lambda) - \lambda = (1 + (t-1)\lambda)^d.\]

Since for $I_s$ only the $t=s$ and $t=s-1$ terms in the sum are nonzero, the corresponding constraint becomes
\begin{align}\label{eqn:dual-recursion}
    \Lopt P_{I_s} + \Lopt_{s-1} \left(\frac{Q_s}{1 + (s-1)\lambda}-1\right) - \Lopt_s \left(\frac{r}{s} - 1\right) = 1,
\end{align}
where we take the convention that $\Lopt_s = 0$ whenever $s \leq 0$ or $s \geq r$. This gives a system of linear equations for the dual variables which clearly has a unique solution. 

We can rewrite this (for $0\leq s < r$) as
\begin{align}
  \Lopt_s = \pfrac{s}{r-s} \left[\Lopt_{s-1}\left(\frac{Q_t}{1-(s-1)\lambda} - 1\right) + \Lopt(Q_s + \lambda) - 1\right].
\end{align}

\begin{claim}
  The solution to the recurrence $a_t = f_t a_{t-1} + g_t$ with $f_t \neq 0$ is
  \[a_t = \left(\prod_{k=1}^t f_k\right)\left(a_0 + \sum_{m=1}^t \frac{g_m}{\prod_{k=1}^m f_k}\right).\]
\end{claim}

We can use this to give an explicit formula for the $\Lopt_t$s, using
\[a_t = \Lopt_t,\qquad
 f_t = \pfrac{t}{r-t}\left(\frac{Q_t}{1-(t-1)\lambda} - 1\right),\qquad
 g_t = \pfrac{t}{r-t}\Lopt(Q_t + \lambda).
\]
The only hitch here is that by this definition we have $f_1 = 0$. However, since $f_1$ is only ever used to multiply by $\Lambda_0 = 0$ we can actually set it to whatever we like. In this case it is easiest to set $f_1 = 1$. Then the formula is
\[\Lopt_t = \left(\prod_{k=2}^t f_k\right)\left(\sum_{m=1}^t \frac{g_m}{\prod_{k=2}^m f_k}\right).\]
Furthermore,
\[\prod_{k=2}^t f_k = \frac{\prod_{k=1}^{t-1} ((1+k\lambda)^{d-1}-1)}{\binom{r-1}{t}} = \frac{t}{r-t}\pfrac{v(1)}{v(t)}.\]
Plugging this in above gives (for $1\leq t < r$)
\begin{align}\label{eqn:Lambdas}
  \Lopt_t &= \frac{t}{r-t} \sum_{s=1}^t \frac{v(s)}{v(t)}(\Lopt P_{I_s}(\lambda) - 1)\\
  &= \frac{1}{Z} \frac{t}{r-t}\sum_{s=1}^t \frac{v(s)}{v(t)} \sum_{i=1}^r v(i) (Q_s - Q_i).
\end{align}

This formula of course fails for $t = r$ (because $f_r$ is undefined), but plugging $\Lopt_r = 0$ into \eqref{eqn:dual-recursion} for $s=r$ allows us to solve for $\Lopt$ and verify that it is equal to $\alpha^*/\lambda$ from the primal solution, as expected from complementary slackness.

We must show that the setting of the dual variables $\Lopt$ and $\Lopt_t$ is dual-feasible. In particular, for every neighborhood configuration $C$, we must show
{\small
\begin{align}
    \Lopt P_C(\lambda) + \sum_t \Lopt_t \eta_{t+1}(C) \frac{P_C(\lambda) - \lambda}{1+t\lambda}
    &\geq 1 + \sum_t \Lopt_t \eta_{t+1}(C) + \sum_t \Lopt_t \eta_t(C)\left(\frac{r}{t} - 1\right)\\
    &= 1 + \sum_t \eta_t(C) \left[\Lopt_{t-1} + \Lopt_t \left(\frac{r}{t} - 1\right)\right].
\end{align}
}

Recalling the notation $Q_t = P_{I_t}(\lambda) - \lambda$, we will also slightly abuse this notation by writing $Q_\eta = P_{C}(\lambda) - \lambda = \prod_t Q_t^{\eta_t}$ (where $\eta = (\eta_t(C))_{t=1}^r$; we may also use this second formula to define $Q_\eta$ for \emph{any} $\eta\in \R^r$). Substituting
	\[\Lopt_{t-1} + \Lopt_t \left(\frac{r}{t} - 1\right) = \Lopt P_{I_t}(\lambda) + \Lopt_{t-1} \frac{Q_t}{1 + (t-1)\lambda} - 1\]
	from \eqref{eqn:dual-recursion}, the constraint for $C$ becomes
	$$ \Lopt P_C(\lambda) + \sum_t \eta_{t+1} \frac{\Lopt_t Q_\eta}{1+t\lambda}\geq 1 + \sum_t \eta_t \left[\Lopt P_{I_t}(\lambda) + \frac{\Lopt_{t-1} Q_t}{1 + (t-1)\lambda} - 1\right]$$
	or equivalently
	$$\sum_t \eta_t \left[\Lopt P_C(\lambda) + \frac{\Lopt_{t-1} Q_\eta}{1+(t-1)\lambda}\right]\geq \sum_t \eta_t \left[\Lopt P_{I_t}(\lambda) + \frac{\Lopt_{t-1} Q_t}{1 + (t-1)\lambda}\right]$$
	By substituting $P_C(\lambda)=Q_\eta+\lambda$ and $P_{I_t}(\lambda)=Q_t+\lambda$ in the left and right hand side respectively, it reduces to
	$$\sum_t \eta_t Q_\eta \left[\Lopt + \frac{\Lopt_{t-1}}{1 + (t-1)\lambda}\right]\geq \sum_t \eta_t Q_t \left[\Lopt + \frac{\Lopt_{t-1}}{1 + (t-1)\lambda}\right].$$
Finally, we can write the constraints simply as a slack constraint
\begin{align}\label{eqn:slack-constraint}
  \sum_t \eta_t (Q_{\eta} - Q_t) \left(\Lopt + \frac{\Lopt_{t-1}}{1 + (t-1)\lambda}\right) \geq 0 
\end{align}
for all convex combinations $\eta$ such that $d \eta$ is integral.

Recalling the formulas \eqref{eqn:opt} and \eqref{eqn:Lambdas} for $\Lopt = \alpha^*/\lambda$ and $\Lopt_t$ and multiplying through by the nonnegative common denominator $Z$, the constraint \eqref{eqn:slack-constraint} expands to
{\small
\begin{align}
  \sum_{t=1}^{r}\frac{\eta_t (Q_\eta - Q_t)}{1+(t-1)\lambda}\sum_{i=1}^r v(i) \left(1+(t-1)\lambda + \Big(\frac{t-1}{r-t+1}\Big)\sum_{s=1}^{t-1} \frac{v(s)}{v(t-1)}(Q_s-Q_i)\right)
  \geq 0
\end{align}}

We would like to show that
\begin{align}
  S(\eta) \defeq \sum_t \eta_t c_t (Q_{\eta} - Q_t) \geq 0 
\end{align}
for every convex combination $\eta$, where
\begin{align}
  c_t \defeq \Lopt + \frac{\Lopt_{t-1}}{1 + (t-1)\lambda}.
\end{align}

Since we have equality (by construction) when $\eta$ is a basis vector, it suffices to show that the only local minima of $S(\eta)$ on the simplex are at its vertices.

If $\eta$ is not a vertex of the simplex, then there is a vector $u$ such that the line segment $[\eta - u, \eta + u]$ is contained in the simplex. If $\eta$ is also a local minimum of $S$, then for every such $u$ the univariate function $\hat S(x) = S(\eta + xu)$ has a local minimum on $[-1,1]$ at $x = 0$, so we must have $\hat S'(0) = 0$ and $\hat S''(0) > 0$. Set $\hat Q(x) = Q_{\eta + xu}$ to get
	$$\hat S(x)= \sum_t c_t (\eta_t + u_t x)(\hat Q(x) - Q_t).$$
	To compute the derivatives of $S(x)$, we need the derivatives of $\hat Q(x)$. Since
	$$\ln \hat Q(x)=\sum_t (\eta_t + u_t x) \ln Q_t,$$
	we have
	$$\frac{d}{dx}[\ln \hat Q(x)] = \frac{\hat Q'(x)} {\hat Q(x)} = \sum_t u_t \ln Q_t = \ln Q_u.$$
	Thus
	$$\hat Q'(x)=\hat Q(x) \ln Q_u$$
	and
	$$\hat Q''(x)=\hat Q(x) (\ln Q_u)^2.$$
	Using the above relations, we get
	$$\hat S'(x) = \sum_t c_t \left[u_t (\hat Q(x) - Q_t) + (\eta_t + u_t x) \hat Q'(x)\right]$$
	and
	$$\hat S''(x)=\hat Q(x) \ln Q_u \sum_t c_t[u_t(2 + x\ln Q_u) + \eta_t \ln Q_u]$$
	
	So at a non-vertex local minimum we would have 
	\begin{align*}
		0 = \hat S'(0)= \sum_t c_t \left[u_t(Q_\eta - Q_t) + \eta_t Q_\eta \ln Q_u\right]
	\end{align*}
	and
	\begin{align*}
		0 \leq \hat S''(0)
		&= \sum_t c_t \left[2 u_t Q_\eta \ln Q_u + \eta_t Q_\eta \ln^2 Q_u\right]\\
		&= \ln Q_u \sum_t c_t \left[2 u_t Q_\eta + u_t (Q_t - Q_\eta)\right]\\
		&= \ln Q_u \sum_t c_t u_t (Q_t + Q_\eta)\,,
	\end{align*}
which implies
\begin{align}\label{eqn:localmin}
  \sum_t c_t u_t (Q_t + Q_\eta) \geq 0
\end{align}
for \emph{every} $u\in \R^r$ with $\sum_t u_t = 0$, $Q_u \geq 1$, and $\supp(u) \subseteq \supp(\eta)$. 

In particular, if $i = \max(\supp(\eta))$ and $j = \min(\supp(\eta))$ it suffices to show that the sum in \eqref{eqn:localmin} is negative for $u = e_i - e_j$ (where $e_k$ is the $k$th basis vector), i.e., that
\[c_i(Q_i + Q_\eta) < c_j (Q_j + Q_\eta).\]
Since $Q_\eta$ can be anything between $Q_j$ and $Q_i$ and the $c_k$s are (as we will show) decreasing in $k$, this is the same as showing that
\[c_i (Q_i + Q_j) < 2 c_j Q_j\,,\]
whenever $i > j$.

We will show that this is true when $r=3$. (When $r\geq 4$ it does not seem to be the case that $\hat S''(0) - \ln Q_u \hat S'(0) > 0$ for some $u = e_i - e_j$).

Specializing to the case $r=3$, 
\begin{align}
  Z c_1 &= (1+\lambda)^{d-1}(1+2\lambda)^{d-1} + (1+2\lambda)^{d-1}-(1+\lambda)^{d-1}\\
  Z c_2 &= \frac{3(1+2\lambda)^{d-1} - 1}{2(1+\lambda)}\\
  Z c_3 &= \frac{2(1+\lambda)^{d-1} - 1}{1+2\lambda}\,.
\end{align}
We must show that $c_1 > c_2 > c_3 > 0$ and that
\[c_i (Q_i + Q_j) < 2 c_j Q_j\,,\]
whenever $i>j$.
\begin{claim}
  $c_1 > c_2 > c_3$.
\end{claim}
\begin{proof}
  To see the first inequality it suffices to see that
  \begin{align}
    (1+\lambda)^d((1+2\lambda)^{d-1}-1) + (1+\lambda)(1+2\lambda)^{d-1} 
    &> 2(1+2\lambda)^{d-1}-1\\
    &> \frac32(1+2\lambda)^{d-1}-\frac12.
  \end{align}
  For the second it suffices to show that
  \begin{align}
    3(1+2\lambda)^d - (1+2\lambda)
    &> 4(1+\lambda)^d - 2(1+\lambda).
  \end{align}
  Indeed, this is true termwise as polynomials in $\lambda$:
  \begin{align}
    2 + (6d-2)\lambda + \sum_{k=2}^d 3\binom{d}{k} 2^k \lambda^k
    &> 2 + (4d-2)\lambda + \sum_{k=2}^d 4\binom{d}{k} \lambda^k.
  \end{align}
  since $3(2^k) \geq 4$ when $k \geq 2$.
  
  Finally, it is clear from inspection that $c_3 > 0$.
\end{proof}

\begin{claim}
  $c_i(Q_i + Q_j) < 2 c_j Q_j$ whenever $1\leq j < i \leq 3$.
\end{claim}
\begin{proof}
  For $i=2$ and $j=1$ we must show that
  \begin{align}
    (3(1+2\lambda)^{d-1}-1)((1+\lambda)^d + 1)
    &< 4 (1+\lambda)^d((1+2\lambda)^{d-1}-1) + 4 (1+\lambda)(1+2\lambda)^{d-1}).
  \end{align}
  This simplifies to showing
  \begin{align}
    (1+\lambda)^d(1+2\lambda)^{d-1} + (1+2\lambda)^d + 2\lambda(1+2\lambda)^{d-1} - 3(1+\lambda)^d + 1 > 0.
  \end{align}
  As before, this is true termwise as a polynomial in $\lambda$:
  \begin{align}
    &(1+\lambda)^d(1+2\lambda)^{d-1} + (1+2\lambda)^d + 2\lambda(1+2\lambda)^{d-1} - 3(1+\lambda)^d + 1\\
    &= \lambda(d+2(d-1)+2d-3d) + \sum_{k\geq 2} \lambda^k \parens[\bigg]{\textstyle\sum_\ell \binom{d}{k-\ell}\binom{d-1}{\ell}2^\ell + \binom{d}{k}2^k + \binom{d-1}{k-1} 2^k - 3 \binom{d}{k}}\\
    &>0\,,
  \end{align}
  since $2^k > 3$ when $k\geq 2$.
  
  For $i=3$ and $j=1$, we must show that
  \begin{align}
    (2(1+\lambda)^{d-1} - 1)((1+2\lambda)^d + 1)
    &< 2((1+\lambda)^{d-1}+1)(1+2\lambda)^d - 2(1+\lambda)^{d-1}(1+2\lambda).
  \end{align}
  This simplifies to 
  \begin{align}
    3(1+2\lambda)^d - 4(1+\lambda)^d + 1 > 0,
  \end{align}
  which is once again true termwise:
  \begin{align}
    3(1+2\lambda)^d - 4(1+\lambda)^d + 1
    &= \sum_{k\geq 1} \lambda^k \parens[\Big]{3\binom{d}{k}2^k - 4\binom{d}{k}} > 0
  \end{align}
  since $3(2^k) > 4$ when $k\geq 1$.
  
  For $i=3$ and $j=2$, we must show that
  \begin{align}
    (2(1+\lambda)^d - (1+\lambda))((1+2\lambda)^d + (1+\lambda)^d) 
    &< (3(1+2\lambda)^d - (1+2\lambda))(1+\lambda)^d.
  \end{align}
  This simplifies to
  \begin{align}
    (1+2\lambda)^d - (1+\lambda)^{d-1}\parens[\Big]{2(1+\lambda)^d - (1+2\lambda)^d + \lambda} > 0,
  \end{align}
  so it suffices to see that
  \begin{align}
    2(1+\lambda)^d - (1+2\lambda)^d = \sum_{k\geq 0} (2 - 2^k) \lambda^k < 1.
  \end{align}
\end{proof}

We conclude by noting that from~\eqref{eqn:opt} the optimum of the LP can be written as~\eqref{eqAlpaFormula}, and thus we have proved Proposition~\ref{thmoccupancy3}.

\section{Conclusions}
In this paper we conjectured some general upper bounds on the number of independent sets in uniform, regular, linear hypergraphs (Conjectures~\ref{conj:strong} and~\ref{conj:linear}) and using the occupancy method proved new bounds in some cases (Theorems~\ref{thm:3unif} and~\ref{thm:weak}).  One immediate direction for future work would be to remove the cross-edge free assumption in these results; as far as we know it is unnecessary, but configurations with cross-edges significantly complicate the analysis of the linear programming relaxations used in the proofs (but in principle this can be done, see~\cite{davies2015}).  Another direction to pursue would be to find a simpler analysis of the linear programming relaxations (or a different set of constraints) that might generalize Theorem~\ref{thm:3unif} to all $r \ge 3$.

\subsection{Non-linear hypergraphs}

While we have focused on linear hypergraphs here, there  are many interesting open questions about independent sets in general (non-linear) hypergraphs.  

\begin{question}
\label{QNonLinear}
For $r, d \ge2$, $2 \le k \le r$, which $r$-uniform, $d$-regular hypergraph $H$ maximizes the quantity
\[ \frac{1}{|V(H)|} \log i_k (H) \, ?\]
\end{question}
By analogy with the graph case ($r=2$) a first guess would be that perhaps the complete $r$-partite hypergraph is the maximizer.  However, along with several other new results in the non-linear case, Balogh, Bollob\'as, and Narayanan~\cite{baloghHyper} have recently shown that this is not true in general by finding a better construction.

\subsection{Lower bounds}
Finally, while we have focused exclusively on upper bounds on the number of independent sets in hypergraphs in this paper, there are many interesting questions about lower bounds on both the maximum size of independent sets and the number of independent sets in various classes of hypergraphs.

Ajtai, Koml{\'o}s, Pintz, Spencer, and Szemer{\'e}di~\cite{ajtai1982extremal} proved a lower bound on the maximum size of a weak independent set in a uniform hypergraph of girth at least $5$ and a given average degree.  Duke, Lefmann, and R{\"o}dl proved a similar lower bound under the weaker assumption that the hypergraph is linear~\cite{duke1995uncrowded} (instead of girth $\ge 5$), and Cooper, Dutta, and Mubayi~\cite{cooper2014counting} proved a lower bound  on the number of weak independent sets in a uniform, linear hypergraph of a given average degree.

The occupancy method has been used to sharpen the lower bound on the number of independent sets in a triangle-free graph~\cite{davies2016averagen}, and one could ask if improvements via the same technique are possible in the case of hypergraphs.  

\subsection{Non-regular hypergraphs}
Finally let us mention that one can ask for the maximum number of independent sets in a hypergraph with  a given number of vertices and edges (not necessarily regular), and in this case Cutler and Radcliffe have determined that the maximizer is the `lexicographic hypergraph'~\cite{Cutler2012}.  The structure of the maximizing hypergraph and the techniques employed are significantly different than the regular case studied here.

\section*{Acknowledgements}
We thank the anonymous referees for very helpful comments on the presentation of the paper.   We thank Bhargav Narayanan for bringing~\cite{baloghHyper} to our attention. 

\bibliographystyle{abbrv}
\bibliography{hypergraphmatchings}

\begin{thebibliography}{10}

\bibitem{ajtai1982extremal}
M.~Ajtai, J.~Koml{\'o}s, J.~Pintz, J.~Spencer, and E.~Szemer{\'e}di.
\newblock Extremal uncrowded hypergraphs.
\newblock {\em Journal of Combinatorial Theory, Series A}, 32(3):321--335,
  1982.

\bibitem{balobanov2018number}
A.~Balobanov and D.~A. Shabanov.
\newblock On the number of independent sets in simple hypergraphs.
\newblock {\em Mathematical Notes}, 103(1-2):33--41, 2018.

\bibitem{baloghHyper}
J.~Balogh, B.~Bollob{\'a}s, and B.~Narayanan.
\newblock Counting independent sets in regular hypergraphs.
\newblock {\em Journal of Combinatorial Theory, Series A}, 180:105405, 2021.

\bibitem{balogh2015independent}
J.~Balogh, R.~Morris, and W.~Samotij.
\newblock Independent sets in hypergraphs.
\newblock {\em Journal of the American Mathematical Society}, 28(3):669--709,
  2015.

\bibitem{Cohen2015}
E.~Cohen, W.~Perkins, and P.~Tetali.
\newblock On the {W}idom--{R}owlinson occupancy fraction in regular graphs.
\newblock {\em Combinatorics, Probability and Computing}, 26(2):183--194, 2017.

\bibitem{cooper2014counting}
J.~Cooper, K.~Dutta, and D.~Mubayi.
\newblock Counting independent sets in hypergraphs.
\newblock {\em Combinatorics, Probability and Computing}, 23(4):539--550, 2014.

\bibitem{Cutler2012}
J.~Cutler and A.~J. Radcliffe.
\newblock {Hypergraph independent sets}.
\newblock {\em Comb. Probab. Comput.}, 22(01):9--20, oct 2012.

\bibitem{davies2015}
E.~Davies, M.~Jenssen, W.~Perkins, and B.~Roberts.
\newblock Independent sets, matchings, and occupancy fractions.
\newblock {\em Journal of the London Mathematical Society}, 96(1):47--66, 2017.

\bibitem{PottsExtremes}
E.~Davies, M.~Jenssen, W.~Perkins, and B.~Roberts.
\newblock Extremes of the internal energy of the {P}otts model on cubic graphs.
\newblock {\em Random Structures \& Algorithms}, 53(1):59--75, 2018.

\bibitem{davies2016averagen}
E.~Davies, M.~Jenssen, W.~Perkins, and B.~Roberts.
\newblock On the average size of independent sets in triangle-free graphs.
\newblock {\em Proceedings of the American Mathematical Society},
  146(1):111--124, 2018.

\bibitem{duke1995uncrowded}
R.~A. Duke, H.~Lefmann, and V.~R{\"o}dl.
\newblock On uncrowded hypergraphs.
\newblock {\em Random Structures \& Algorithms}, 6(2-3):209--212, 1995.

\bibitem{galvin2004weighted}
D.~Galvin and P.~Tetali.
\newblock On weighted graph homomorphisms.
\newblock {\em DIMACS Series in Discrete Mathematics and Theoretical Computer
  Science}, 63:97--104, 2004.

\bibitem{Kahn2001}
J.~Kahn.
\newblock {An entropy approach to the hard-core model on bipartite graphs}.
\newblock {\em Comb. Probab. Comput.}, 10(3):219--237, jun 2001.

\bibitem{korshunov1983number}
A.~Korshunov and A.~Sapozhenko.
\newblock The number of binary codes with distance 2.
\newblock {\em Problemy Kibernet}, 40:111--130, 1983.

\bibitem{Ordentlich2004}
E.~Ordentlich and R.~M. Roth.
\newblock {Independent sets in regular hypergraphs and multidimensional
  runlength-limited constraints}.
\newblock {\em SIAM J. Discret. Math.}, 17(4):615--623, 2004.

\bibitem{girthCubicA}
G.~Perarnau and W.~Perkins.
\newblock Counting independent sets in cubic graphs of given girth.
\newblock {\em Journal of Combinatorial Theory, Series B}, 133:211--242, 2018.

\bibitem{sapozhenko1989number}
A.~A. Sapozhenko.
\newblock The number of antichains in ranked partially ordered sets.
\newblock {\em Diskretnaya Matematika}, 1(1):74--93, 1989.

\bibitem{saxton2015hypergraph}
D.~Saxton and A.~Thomason.
\newblock Hypergraph containers.
\newblock {\em Inventiones mathematicae}, 201(3):925--992, 2015.

\bibitem{zhao2010number}
Y.~Zhao.
\newblock The number of independent sets in a regular graph.
\newblock {\em Combinatorics, Probability and Computing}, 19(02):315--320,
  2010.

\bibitem{yufeiSurvey}
Y.~Zhao.
\newblock Extremal regular graphs: independent sets and graph homomorphisms.
\newblock {\em The American Mathematical Monthly}, 124(9):827--843, 2017.

\end{thebibliography}

\end{document}